\documentclass[11pt]{amsart}
\usepackage{amsmath,amsthm,amsfonts,amscd,amssymb,eucal,latexsym,mathrsfs,stmaryrd,enumitem,bm,mathtools,microtype}

\newtheorem{theorem}{Theorem}[section]

\newtheorem{lemma}[theorem]{Lemma}
\newtheorem{proposition}[theorem]{Proposition}

\theoremstyle{definition}
\newtheorem{definition}[theorem]{Definition}

\newtheorem{question}[theorem]{Question}

\newcounter{theoremintro}
\newtheorem{theoremi}[theoremintro]{Theorem}

\newcommand{\sN}{{\mathcal N}}
\newcommand{\cU}{{\mathcal U}}

\newcommand{\cZ}{{\mathcal Z}}

\newcommand{\Zb}{{\mathbb Z}}
\newcommand{\Nb}{{\mathbb N}}

\newcommand{\sM}{{\mathcal M}}

\newcommand{\sP}{{\mathcal P}}

\newcommand{\sV}{{\mathcal V}}

\newcommand{\sZ}{{\mathcal Z}}

\newcommand{\id}{{\rm id}}

\numberwithin{equation}{section}

\allowdisplaybreaks

\begin{document}

\title{The small boundary property in products}
\author{David Kerr}
\address{David Kerr,
Mathematisches Institut,
University of M{\"u}nster,
Einsteinstr.\ 62,
48149 M{\"u}nster, Germany}
\email{kerrd@uni-muenster.de}

\author{Hanfeng Li}
\address{Hanfeng Li,
Department of Mathematics, SUNY at Buffalo, Buffalo, NY 14260-2900, USA
}
\email{hfli@math.buffalo.edu}

\date{June 17, 2024}

\begin{abstract}
For a continuous action $G\curvearrowright X$ of a countable group on a compact metrizable space we show that 
the following are equivalent:
(i) the action $G\curvearrowright X$ has the small boundary property and no finite orbits,
(ii) for every continuous action $H\curvearrowright Y$ of a countable group on a compact metrizable space,
the product action $G\times H\curvearrowright X\times Y$ has the small boundary property.
In particular, (ii) is automatic when $G$ is infinite and the action $G\curvearrowright X$ is minimal
and has the small boundary property.
The argument relies on a small boundary version of the Urysohn lemma.
\end{abstract}

\maketitle

\section{Introduction}

Over the last decade several analogies have emerged between regularity properties in the classification program
for nuclear C$^*$-algebras and properties of the topological dynamical systems (mainly
group actions on compact metrizable spaces)
that provide C$^*$-algebra theory with many of its basic models via the crossed product construction.
Prominent among these pairings (see \cite{Ker20,KerSza20,CasEviTikWhi22}) are
\begin{itemize}
\item $\sZ$-stability and almost finiteness, 

\item the uniform McDuff property (or, equivalently in the simple separable nuclear setting, 
uniform property Gamma \cite{CasEviTikWhi22}) and almost finiteness in measure.
\end{itemize}
Almost finiteness and almost finiteness in measure are both Rokhlin-type tiling properties that ask for the existence of 
open towers with F{\o}lner shapes and levels of small diameter \cite{Ker20,KerSza20}. 
The difference between the two is whether
the complement of the towers is small topologically (via a notion of subequivalence using open covers) 
or uniformly on invariant Borel probability measures. 
Due to the F{\o}lner requirement these properties necessitate that the acting group be amenable, while $\sZ$-stability and
the uniform McDuff property are independent of amenability/nuclearity, so that nuclearity
is implicitly being assumed when making the above pairings.
These analogies run quite far 
and, together with the property of dynamical comparison, have been instrumental in advancing the project of classifying
C$^*$-crossed products. In each case
the backwards implication holds when passing from a minimal action of a countably infinite group on 
a compact metrizable space to its corresponding crossed product \cite{Ker20,KerSza20}, and it is conceivable that
one even has an equivalence in the nuclear setting. Indeed $\sZ$-stability is known to fail
for the typical examples of minimal actions of countable amenable groups that are not almost finite \cite{GioKer10,HirPhi22}.

At the same time, the structural homologies supporting these relationships 
are far from perfect and the discrepancies at play are closely tied up with a number of 
subtle issues in topological dynamics.
One can particularly appreciate some of the sticking points
when looking at behaviour under taking products, or tensor products in the case of C$^*$-algebras. 
A product $G\times H \curvearrowright X\times Y$ of two actions $G\curvearrowright X$
and $H\curvearrowright Y$ on compact spaces naturally gives rise to a minimal C$^*$-tensor product
$C(X\times Y)\rtimes_\lambda (G\times H) \cong (C(X)\rtimes_\lambda G) \otimes_{{\rm min}} (C(Y)\rtimes_\lambda H)$
of the reduced crossed products of the factors. Now $\sZ$-stability passes 
to the minimal tensor product with any other C$^*$-algebra 
directly by definition, without any nuclearity assumptions, 
and so $C(X\times Y)\rtimes_\lambda (G\times H)$ will have
this property as soon as one of the tensor factors does. 
The same can be said for the uniform McDuff property granted that one works within the relevant universe of unital tracial
C$^*$-algebras and actions admitting invariant Borel probability measures.
On the dynamical side, 
the tower structure that comes from assuming one of the factors is almost finite or almost finite in measure 
will naturally pass to the product through amplification but one will lose both the F{\o}lner condition on the tower shapes
and the small diameter condition on the tower levels, and it is not at all clear 
how to engineer these features unless one has some subdivisibility mechanism at hand. 
It turns out however that one can establish the desired subdivisibility under such circumstances 
(assuming $G$ and $H$ are infinite)
in the form of the small boundary property (SBP). In the case that the actions $G\curvearrowright X$
and $H\curvearrowright Y$ are free this was accomplished independently by Kopsacheilis, Liao, Tikuisis, and Vaccaro 
\cite{KopLiaTikVac24} and Elliott and Niu \cite{EllNiu24}
by passing through a version of property Gamma for inclusions at the C$^*$-algebra level.
The SBP is actually equivalent, via Ornstein--Weiss tiling technology, to almost finiteness in measure
when the acting group is amenable and the action is free \cite{KerSza20},
and so one obtains the conclusion that if $G\curvearrowright X$
and $H\curvearrowright Y$ are free and one of them is almost finite in measure then 
the product $G\times H \curvearrowright X\times Y$ is almost finite in measure.
Kopsacheilis, Liao, Tikuisis, and Vaccaro showed moreover, by passing through 
a version of tracial $\sZ$-stability for inclusions and making use of dynamical comparison
and almost finiteness in measure via \cite{KerSza20,LiaTik22}, that the corresponding statement also holds 
for almost finiteness \cite{KopLiaTikVac24}.

By definition, the SBP requires that the topology admit a basis of open sets whose boundaries are null 
for every invariant Borel probability measure, a condition which is meaningful whether or not the acting group is amenable.
In the amenable setting it implies mean dimension zero \cite{LinWei00} and is equivalent to it when 
the group is $\Zb^d$ and the action has the marker property (e.g., when it is free and minimal) \cite{Lin99,GutLinTsu16}, 
and also more generally whenever
the action has the uniform Rokhlin property \cite{Niu21}. 
It is a simple observation that a product $G\times H \curvearrowright X\times Y$ of actions of infinite amenable groups always
has mean dimension zero, and so
when $G$ and $H$ are both finitely generated free Abelian groups and the actions are free and minimal
one can draw the much stronger and quite surprising conclusion that the product action
always has the SBP and hence, via \cite{Nar22}, is always almost finite. 

The uniform McDuff property and the SBP are quite different in nature at the technical level, 
and perhaps it is unique to amenability, via the connection to almost finiteness in measure, 
that one is able to compare them in that setting.
Nevertheless, we show in the present paper that the SBP does continue to share the same permanence behaviour in products 
beyond the amenable realm. Our approach is completely different from and more direct than the ones 
in \cite{KopLiaTikVac24,EllNiu24}.
It also yields a more general statement in the amenable case that applies in particular to all minimal actions
without the hypothesis of freeness that is present in all of the arguments above
involving the SBP and that is also, at least to a certain extent,
built into the definitions of almost finiteness and almost finiteness in measure
(these concepts are also viable for actions that are essentially free \cite{GarGefGesKopNar24}, 
but there are minimal topologically free actions of amenable groups that do not have this property \cite{Jos24}).

\begin{theoremi}\label{T-main}
Let $G\curvearrowright X$ be a continuous action of a countable group on a compact metrizable space. Then the following are equivalent:
\begin{enumerate}
\item the action $G\curvearrowright X$ has the SBP and no finite orbits,
\item for every continuous action $H\curvearrowright Y$ of a countable group on a compact metrizable space, the product action $G\times H\curvearrowright X\times Y$ has the SBP.
\end{enumerate}
In particular, condition (2) holds whenever $G$ is infinite and the
action $G\curvearrowright X$ is minimal and has the SBP.
\end{theoremi}

The arguments in our proof of Theorem~\ref{T-main} do not actually themselves have much 
to do with dynamics, which really only enters at the point where we need to know that ergodic measures for the product action
are product measures (Lemma~\ref{L-ergodic for product}).
The crucial tool is a small boundary version of the Urysohn lemma, which we establish in Section~\ref{S-Urysohn}
as Proposition~\ref{P-SBP to function}.
This is related to, but somewhat different from, the recursive constructions used in \cite{ShuWei91} and \cite{Lin95}
to control the value of entropy in factors of actions with the SBP, which was the original motivation 
for the introduction of the SBP.
In Section~\ref{S-product actions} we establish Theorem~\ref{T-main} by applying our Urysohn-type lemma 
to produce sets in the product with small boundary.
For this we rely on a characterization of the SBP in terms of pairs of points (Proposition~\ref{P-small separation})
that, as an offshoot, suggests how the analysis of small boundaries can be localized  
in the spirit of entropy pairs (see \cite{KerLi16}) and mean dimension pairs \cite{GarGut24}. 
We briefly indicate in Section~\ref{S-SBP pairs} how this localization can be set up through the notion of SBP pair.

That amenability and freeness turn out not to be relevant for Theorem~\ref{T-main}
is in the spirit of \cite{KerKopPet24}, where it is shown that the SBP holds whenever
the space of invariant Borel probability measures is a Bauer simplex with finite-dimensional extreme boundary.
In fact the only essential way in which dynamics gets used in both \cite{KerKopPet24} and the present paper,
besides the affine function realization in \cite{KerKopPet24},
is in the application of the comparison property for ergodic probability-measure-preserving actions in order
to assemble disjoint sets of equal measure into a tower whose shape is a subset of the full group, 
an operation that requires neither amenability nor freeness.

Finally we mention one consequence of Theorem~\ref{T-main} for the classifiability of crossed products.
Let $G\curvearrowright X$ and $H\curvearrowright Y$ be minimal actions of 
finitely generated groups of polynomial growth on compact metrizable spaces and suppose that
one of them has the SBP. By \cite{Ker20,Nar22} and the previously known permanence of the SBP 
under products discussed above, if these two actions are both free then 
the crossed product $C(X\times Y)\rtimes_\lambda (G\times H)$ 
of the product action $G\times H \curvearrowright X\times Y$ falls under the
scope of the classification theorem for simple, separable, unital, nuclear, $\cZ$-stable $C^*$-algebras satisfying the UCT
(a discussion of which can be found for example in the introduction to \cite{CarGabSchTikWhi23}). 
Together with \cite[Corollary D]{GarGefGesKopNar24}, Theorem~\ref{T-main} shows that
this classifiability is in fact equivalent to both of the factor actions being topologically free
(a necessary condition for ensuring that the C$^*$-algebra is simple).
\medskip

\noindent{\it Acknowledgements.}
The authors were supported by the Deutsche Forschungsgemeinschaft
(DFG, German Research Foundation) under Germany's Excellence Strategy EXC 2044-\linebreak 390685587,
Mathematics M{\"u}nster: Dynamics--Geometry--Structure, and 
by the SFB 1442 of the DFG.

\section{A small boundary Urysohn lemma} \label{S-Urysohn}

Throughout $X$ is a compact metrizable space and $\sN$ a closed subset of 
the space $\sM(X)$ of Borel probability measures on $X$ equipped with the weak$^*$ topology.

\begin{definition} \label{D-small}
We say a Borel subset $Y$ of $X$ is {\it $\sN$-small} if $\sup_{\mu\in \sN}\mu(Y)=0$. We say $X$ has  
the {\it $\sN$-SBP} ($\sN$-small boundary property) if there is a basis for its topology 
consisting of open sets whose boundary is $\sN$-small.
\end{definition}

Our goal is to establish in Proposition~\ref{P-SBP to function} a Urysohn lemma for the $\sN$-SBP.

\begin{proposition} \label{P-small separation}
The following are equivalent:
\begin{enumerate}
\item $X$ has the $\sN$-SBP.
\item For any two disjoint closed subsets $A_1$ and $A_2$ of $X$ 
there is an open set $U$ such that $A_1\subseteq U$, $A_2\cap \overline{U}=\emptyset$, and $\partial U$ is $\sN$-small.
\item For any two distinct $x, y\in X$ there is an open set $U\subseteq X$ such that $x\in U$, $y\not\in \overline{U}$ and $\partial U$ is $\sN$-small.
\end{enumerate}
\end{proposition}

\begin{proof} 
(1)$\Rightarrow$(2). 
By (1), for each $x\in A_1$ there is an open neighbourhood $U_x$ of $x$ such that $\overline{U_x}\subseteq X\setminus A_2$ and $\partial U_x$ is $\sN$-small. Since $A_1$ is compact, there is a finite subset $W$ of $A_1$ such that the sets $U_x$ for $x\in W$ cover $A_1$. Set $U=\bigcup_{x\in W}U_x$. Then $U$ is open, $A_1\subseteq U$, and $\overline{U}=\bigcup_{x\in W}\overline{U_x}\subseteq X\setminus A_2$. Note that $\partial U\subseteq \bigcup_{x\in W}\partial U_x$ is $\sN$-small. Thus (2) holds.

(2)$\Rightarrow$(3). Trivial.

(3)$\Rightarrow$(1). 
Let $x\in X$ and let $V$ be an open neighbourhood of $x$. By (3), for each $y\in X\setminus V$ we can find an open neighbourhood $U_y$ of $x$ such that $y\in X\setminus \overline{U_y}$ and $\partial U_y$ is $\sN$-small. Since $X\setminus V$ is compact, there is a finite subset $W$ of $X\setminus V$ such that the sets $X\setminus \overline{U_y}$ for $y\in W$ cover $X\setminus V$. Then $U:=\bigcap_{y\in W}U_y$ is an open neighbourhood of $x$, and
\begin{gather*}
\overline{U}\subseteq \bigcap_{y\in W}\overline{U_y}=X\setminus \bigg(X\setminus \bigcap_{y\in W}\overline{U_y}\bigg)=X\setminus \bigcup_{y\in W}(X\setminus \overline{U_y})\subseteq X\setminus (X\setminus V)=V.
\end{gather*}
Note that $\partial U\subseteq \bigcup_{y\in W}\partial U_y$ is $\sN$-small. Thus (1) holds.
\end{proof}

\begin{lemma} \label{L-small separation}
Suppose that $X$ has the $\sN$-SBP. Let $A$ be an $\sN$-small closed subset of $X$ and let $V$ be an open subset of $X$ such that $A\subseteq V$. Let $\varepsilon>0$. Then there is an open set $U\subseteq X$ such that $A\subseteq U\subseteq \overline{U}\subseteq V$, $\partial U$ is $\sN$-small, and $\mu(U)<\varepsilon$ for all $\mu\in \sN$.
\end{lemma}

\begin{proof} 
Let $\mu\in \sN$. Since $\mu(A)=0$, we can find an open set $U_\mu$ such that $A\subseteq U_\mu\subseteq \overline{U_\mu}\subseteq V$ and $\mu(\overline{U_\mu})<\varepsilon$. Then we can find an open neighbourhood $O_\mu$ of $\mu$ in $\sN$ such that $\nu(\overline{U_\mu})<\varepsilon$ for all $\nu\in O_\mu$.

Since $\sN$ is compact, we can find a finite subset $W$ of $\sN$ such that the sets $O_\mu$ for $\mu\in W$ cover $\sN$. Set $U=\bigcap_{\mu\in W}U_\mu$. Then $U$ is open, $A\subseteq U\subseteq V$, and $\nu(U)<\varepsilon$ for all $\nu\in \sN$. By condition (2) in Proposition~\ref{P-small separation} we can find an open set $U'$ such that $A\subseteq U'\subseteq \overline{U'}\subseteq U$ and $\partial U'$ is $\sN$-small.
\end{proof}

\begin{lemma} \label{L-SBP to function1}
Suppose that $X$ has the $\sN$-SBP. Let $A_0$ and $A_1$ be disjoint closed subsets of $X$, let $\mu\in \sM(X)$ be atomless, 
and let $\varepsilon, C>0$. Then there are an integer $M\ge C$ and open sets $U_1, \dots, U_M \subseteq X$ satisfying the following conditions:
\begin{enumerate}
\item $A_0\subseteq U_1$,
\item $\overline{U_j}\subseteq U_{j+1}$ for all $j=1, \dots, M-1$,
\item $U_M=X\setminus A_1$,
\item $\partial U_j$ is $\sN$-small for all $j=1, \dots, M-1$,
\item $\mu(U_1\setminus A_0)<\varepsilon$,
\item $\mu(U_{j+1}\setminus \overline{U_j})<\varepsilon$ for all  $j=1, \dots, M-1$.
\end{enumerate}
\end{lemma}

\begin{proof}
Take an integer $M\ge C$ such that $\frac{1}{M}\mu(X\setminus (A_0\cup A_1))<\frac{1}{2}\varepsilon$.
Since $\mu$ is atomless and $X$ is metrizable, $(X,\mu )$ is measure-isomorphic to the unit interval with Lebesgue measure
\cite[Theorem~A.20]{KerLi16}
and so we can find a Borel partition $\sP=\{P_1, \dots, P_M\}$ of $X\setminus (A_0\cup A_1)$ such that $\mu(P_j)=\frac{1}{M}\mu(X\setminus (A_0\cup A_1))$ for every $j=1, \dots, M$. For each $1\le j\le M$, take a compact subset $Y_j$ of $P_j$ such that $\mu(P_j\setminus Y_j)<\varepsilon/(2M)$.

By condition (2) in Proposition~\ref{P-small separation}, we can find an open set $U_1$ such that $A_0\cup Y_1\subseteq U_1$, $\overline{U_1}$ is disjoint from $A_1\cup \bigcup_{k=2}^MY_k$, and $\partial U_1$ is $\sN$-small.
Again using condition (2) in Proposition~\ref{P-small separation} we can find an open set $U_2$ such that $\overline{U_1}\cup Y_2\subseteq U_2$, $\overline{U_2}$ is disjoint from $A_1\cup \bigcup_{k=3}^MY_k$, and $\partial U_2$ is $\sN$-small.
Continuing in this way, we produce open sets $U_1, \dots, U_{M-1}$ of $X$ such that $\overline{U_{j-1}}\cup Y_j\subseteq U_j$, $\overline{U_j}$ is disjoint from $A_1\cup \bigcup_{k=j+1}^MY_k$, and $\partial U_j$ is $\sN$-small for all $2\le j\le M-1$.
Set $U_M=X\setminus A_1$. Then conditions (1) to (4) in the lemma statement hold.

Since the sets $U_1\setminus A_0$ and $Y_k$ for $k\in \{2, \dots, M\}$ are pairwise disjoint subsets of 
$X\setminus (A_0\cup A_1)$, we have
\begin{align*}
\mu(U_1\setminus A_0)&\le \mu(X\setminus (A_0\cup A_1))-\sum_{2\le k\le M}\mu(Y_k)\\
&=\mu(X\setminus (A_0\cup A_1))-\sum_{2\le k\le M}\mu(P_k)+\sum_{2\le k\le M}\mu(P_k\setminus Y_k)\\
&=\frac{1}{M}\mu(X\setminus (A_0\cup A_1))+\sum_{2\le k\le M}\mu(P_k\setminus Y_k)\\
&\le \frac{1}{2}\varepsilon+(M-1)\frac{\varepsilon}{2M}<\varepsilon.
\end{align*}
This establishes (5).

Let $1\le j\le M-1$. Then the sets $U_{j+1}\setminus \overline{U_j}$ and $Y_k$ for $k\in \{1, \dots, M\}\setminus \{j+1\}$ are pairwise disjoint subsets of $X\setminus (A_0\cup A_1)$, in which case
\begin{align*}
\mu(U_{j+1}\setminus \overline{U_j})&\le \mu(X\setminus (A_0\cup A_1))-\sum_{\substack{1\le k\le M\\ k\neq j+1}}\mu(Y_k)\\
&=\mu(X\setminus (A_0\cup A_1))-\sum_{\substack{1\le k\le M\\ k\neq j+1}}\mu(P_k)+\sum_{\substack{1\le k\le M\\ k\neq j+1}}\mu(P_k\setminus Y_k)\\
&=\frac{1}{M}\mu(X\setminus (A_0\cup A_1))+\sum_{\substack{1\le k\le M\\ k\neq j+1}}\mu(P_k\setminus Y_k)\\
&\le \frac{1}{2}\varepsilon+(M-1)\frac{\varepsilon}{2M}<\varepsilon.
\end{align*}
This establishes (6).
\end{proof}

\begin{lemma} \label{L-SBP to function2}
Suppose that $\sN$ consists of atomless measures and that $X$ has the $\sN$-SBP. Let $A_0$ and $A_1$ be disjoint closed subsets of $X$, and let $\varepsilon, C>0$. Then there are a positive integer $M\ge C$ and open sets $U_1, \dots, U_M \subseteq X$ satisfying the following conditions:
\begin{enumerate}
\item $A_0\subseteq U_1$,
\item $\overline{U_j}\subseteq U_{j+1}$ for all $j=1, \dots, M-1$,
\item $U_M=X\setminus A_1$,
\item $\partial U_j$ is $\sN$-small for all $j=1, \dots, M-1$,
\item $\mu(U_1\setminus A_0)<\varepsilon$ for all $\mu\in \sN$,
\item $\mu(U_{j+1}\setminus \overline{U_j})<\varepsilon$ for all $\mu\in \sN$ and $j=1, \dots, M-1$.
\end{enumerate}
\end{lemma}

\begin{proof} 
We may assume that $\sN\neq \emptyset$, for otherwise the statement is trivial.

Let $\mu\in \sN$. By Lemma~\ref{L-SBP to function1} there are a positive integer $M_{\mu}\ge C+2$ and open sets $U_{\mu, 1}, \dots, U_{\mu, M_\mu}\subseteq X$ satisfying the following conditions:
\begin{enumerate}[label=(\roman*)]
\item $A_0\subseteq U_{\mu, 1}$,
\item $\overline{U_{\mu, j}}\subseteq U_{\mu, j+1}$ for all $j=1, \dots, M_\mu-1$,
\item $U_{M_\mu}=X\setminus A_1$,
\item $\partial U_{\mu, j}$ is $\sN$-small for all $j=1, \dots, M_\mu-1$,
\item $\mu(U_{\mu, 1}\setminus A_0)<\varepsilon/2$,
\item $\mu(U_{\mu, j+1}\setminus \overline{U_{\mu, j}})<\varepsilon/2$ for all  $j=1, \dots, M_\mu-1$.
\end{enumerate}
Consider the pairwise disjoint collection
\begin{gather*}
\sV_\mu :=\{U_{\mu, 1}\setminus A_0\}\cup \{U_{\mu, j+1}\setminus \overline{U_{\mu, j}}: j=1, \dots, M_\mu-1\}
\end{gather*}
of open subsets of $X\setminus (A_0\cup A_1)$. By (iv) the set
\begin{gather*}
(X\setminus (A_0\cup A_1))\setminus \bigcup \sV_\mu=\bigcup_{j=1}^{M_\mu-1}\partial U_{\mu, j}
\end{gather*}
is $\sN$-small, and by (v) and (vi) we have $\mu(V)<\varepsilon/2$ for all $V\in \sV_\mu$.
By (iv) and the portmanteau theorem, 
the functions sending $\nu$ to $\nu(U_{\mu, j})$ for $j=1, \dots, M_\mu-1$ and $\nu(X\setminus \overline{U_{\mu, M_\mu-1}})$ are continuous on $\sN$.
It follows that for each $V\in \sV_\mu$ the function on $\sN$ sending $\nu$ to $\nu(V)$ is continuous.
We can therefore find an open neighbourhood $O_\mu$ of $\mu$ in $\sN$ such that
\begin{gather*}
\nu(V)<\frac{\varepsilon}{2}
\end{gather*}
for all $\nu\in O_\mu$ and $V\in \sV_\mu$. Enumerate the elements of $\sV_\mu$ as $V_{\mu, 1}, \dots, V_{\mu, M_\mu}$ 
with $V_{\mu, 1}=U_{\mu, 1}\setminus A_0$ and $V_{\mu, M_\mu}=U_{\mu, M_\mu}\setminus \overline{U_{\mu, M_\mu-1}}$. Then $V_{\mu, 1}\cup A_0=U_{\mu, 1}$ and $V_{\mu, M_\mu}\cup A_1=X\setminus \overline{U_{\mu, M_\mu-1}}$ are open.

Since $\sN$ is compact, we can find a finite subset $W$ of $\sN$ such that the sets $O_\mu$ for $\mu\in W$ cover $\sN$.
Set $M=\prod_{\mu\in W}M_\mu\ge C+2$. Then $M\ge 3$. Denote by $\sV$ the join of the sets $\sV_\mu$ for $\mu\in W$, i.e., the family of sets of the form $\bigcap_{\mu\in W}V_\mu$ with $V_\mu\in \sV_\mu$ for all $\mu\in W$. Then $\sV$ consists of $M$ pairwise disjoint open subsets of $X\setminus (A_0\cup A_1)$, some of which may be empty.
Note that the set
\begin{gather*}
(X\setminus (A_0\cup A_1))\setminus \bigcup \sV=\bigcup_{\mu\in W}\Big((X\setminus (A_0\cup A_1))\setminus \bigcup \sV_\mu\Big)
\end{gather*}
is $\sN$-small. Furthermore,
\begin{gather*}
\nu(V)<\frac{\varepsilon}{2}
\end{gather*}
for all $V\in \sV$ and $\nu\in \sN$. Enumerate the elements of $\sV$ as $V_1, \dots, V_M$ with
\begin{gather*}
V_1=\bigcap_{\mu\in W}V_{\mu, 1} \hspace*{4mm}
\text{and}\hspace*{4mm}
V_M=\bigcap_{\mu\in W}V_{\mu, M_\mu}.
\end{gather*}
Then $V_1\cup A_0=\bigcap_{\mu\in W}(V_{\mu, 1}\cup A_0)$ and $V_M\cup A_1=\bigcap_{\mu\in W}(V_{\mu, M_\mu}\cup A_1)$ are open. For each $1<j<M$ the boundary $\partial V_j$ is a subset of $(X\setminus (A_0\cup A_1))\setminus \bigcup \sV$ and hence is $\sN$-small.
Set $U_1=V_1\cup A_0$. Since $V_1\cup A_0$ and $V_M\cup A_1$ are disjoint open sets, $\partial U_1$ is a subset of $(X\setminus (A_0\cup A_1))\setminus \bigcup \sV$ and hence is $\sN$-small.

By Lemma~\ref{L-small separation} we can find an open set $R_1 \subseteq X$ such that $\partial U_1\subseteq R_1\subseteq \overline{R_1}\subseteq X\setminus (A_0\cup A_1)$, $\partial R_1$ is $\sN$-small, and $\mu(R_1)<\varepsilon/2$ for all $\mu\in \sN$. Set $U_2=U_1\cup R_1\cup V_2$. Then $\overline{U_2}=\overline{U_1}\cup \overline{R_1}\cup \overline{V_2}$ is disjoint from $A_1$, and $\overline{U_1}\subseteq U_2$. Note that $\partial U_2\subseteq \partial U_1\cup \partial R_1\cup \partial V_2$ is $\sN$-small. For each $\mu\in \sN$, we have
\begin{gather*}
\mu(U_2\setminus \overline{U_1})\le \mu(R_1\cup V_2)\le \mu(R_1)+\mu(V_2)<\frac{\varepsilon}{2}+\frac{\varepsilon}{2}=\varepsilon.
\end{gather*}

If $M=3$, then setting $U_3=X\setminus A_1$ we have $\mu(U_3 \setminus \overline{U_2})\le \mu(V_3)<\varepsilon/2$ for all $\mu\in \sN$, 
in which case (1) to (6) hold. Thus we may assume that $M>3$.

By Lemma~\ref{L-small separation} we can find an open set $R_2$ of $X$ such that $\partial U_2\subseteq R_2\subseteq \overline{R_2}\subseteq X\setminus (A_0\cup A_1)$, $\partial R_2$ is $\sN$-small, and $\mu(R_2)<\varepsilon/2$ for all $\mu\in \sN$. Set $U_3=U_2\cup R_2\cup V_3$. Then $\overline{U_3}=\overline{U_2}\cup \overline{R_2}\cup \overline{V_3}$ is disjoint from $A_1$, and $\overline{U_2}\subseteq U_3$. Note that $\partial U_3\subseteq \partial U_2\cup \partial R_2\cup \partial V_3$ is $\sN$-small. For each $\mu\in \sN$, we have
\begin{gather*}
\mu(U_3\setminus \overline{U_2})\le \mu(R_2\cup V_3)\le \mu(R_2)+\mu(V_3)<\frac{\varepsilon}{2}+\frac{\varepsilon}{2}=\varepsilon.
\end{gather*}

Continuing in this way, we construct open sets $U_2, \dots, U_{M-1}$ such that (4) holds and $\overline{U_j}\cup V_{j+1}\subseteq U_{j+1}$ and $\mu(U_{j+1}\setminus \overline{U_j})<\varepsilon$ for all $j=1, \dots, M-2$ and $\mu\in \sN$, and $\overline{U_{M-1}}$ is disjoint from $A_1$. Set $U_M=X\setminus A_1$. Then $\overline{U_{M-1}}\subseteq U_M$, and $\mu(U_M\setminus \overline{U_{M-1}})\le \mu(V_M)<\varepsilon/2$ for all $\mu\in \sN$. Therefore (1) to (6) hold.
\end{proof}

\begin{proposition} \label{P-SBP to function}
Suppose that $\sN$ consists of atomless measures and $X$ has the $\sN$-SBP. Let $A_0$ and $A_1$ be disjoint closed subsets of $X$. Then there is a continuous function $f: X\rightarrow [0, 1]$ such that $f=t$ on $A_t$ for $t\in \{0, 1\}$ and $f^{-1}(t)\setminus (A_0\cup A_1)$ is $\sN$-small for all $t\in [0, 1]$.
\end{proposition}

\begin{proof} 
We will construct a sequence $\{Q_n\}_{n\in \Nb}$ of finite subsets of $(0, 1]$ and an open set $U_r \subseteq X$ for each $r\in Q:=\bigcup_{n\in \Nb}Q_n$ such that the following hold:
\begin{enumerate}
\item $1\in Q_n\subseteq Q_{n+1}$ for all $n$,
\item $\overline{U_r}\subseteq U_{r'}$ for all $r<r'$ in $Q$,
\item $A_0\subseteq U_r$ for all $r\in Q$ and $U_1=X\setminus A_1$,
\item $\partial U_r$ is $\sN$-small for each $r\in Q$,
\item setting $U_0=A_0$, for each $n$,  if we list the elements of $Q_n\cup \{0\}$ in increasing order, then for any consecutive $r<r'$ in $Q_n\cup \{0\}$ we have $r'-r\le 2^{-n}$ and $\mu(U_{r'}\setminus \overline{U_r})<2^{-n}$ for all $\mu\in \sN$.\end{enumerate}
Assume for now that we have found such $\{Q_n\}_{n\in \Nb}$ and $U_r$. From (1) and (5) we know that $Q$ is dense in $[0, 1]$. From (2) we know that the boundaries $\partial U_r$ for $r\in Q$ are pairwise disjoint.

Let $x\in X\setminus (A_0\cup A_1)$. If $x\in U_{r'}$ and $x\not\in \overline{U_r}$ for some $r, r'\in Q\cup \{0\}$, then by (2) and (3) we have $r<r'$. Set
\begin{align*}
Q_{>x}&=\{r\in Q\cup \{0\}: x\in U_r\} ,\\
Q_{<x}&=\{r\in Q\cup \{0\}: x\not\in \overline{U_r}\}.
\end{align*}
Then $1\in Q_{>x}$ and $0\in Q_{<x}$.
We claim that $\sup Q_{<x}=\inf Q_{>x}$. Suppose instead that $\sup Q_{<x}<\inf Q_{>x}$. For any $r\in Q$ satisfying $\sup Q_{<x}<r<\inf Q_{>x}$, we have
$x\in \overline{U_r}$ and $x\not\in U_r$, whence $x\in \partial U_r$. Since $Q$ is dense in $(0, 1)$, there are infinitely many $r\in Q$ satisfying $\sup Q_{<x}<r<\inf Q_{>x}$, and the corresponding $\partial U_r$ are disjoint. Thus we obtain a contradiction,
proving our claim.

Now we define $f: X\rightarrow [0, 1]$ by
\begin{align*}
f(x)=\begin{cases}
0 & \quad \text{if } x\in A_0 ,\\
\sup Q_{<x}=\inf Q_{>x} & \quad \text{if } x\not\in A_0\cup A_1 ,\\
1 & \quad \text{if } x\in A_1.\\
\end{cases}
\end{align*}
Then $f=t$ on $A_t$ for $t\in \{0, 1\}$.

Let $x\in X$ and let us check that $f$ is continuous at $x$.
Suppose first that $f(x)\in (0, 1)$. For every $\varepsilon>0$, since $Q$ is dense in $[0, 1]$ we can find $r, r'$ in $Q$ such that $f(x)-\varepsilon<r<f(x)<r'<f(x)+\varepsilon$. Then $x\in U_{r'}\setminus \overline{U_r}$. Note that $r\le f(y)\le r'$ for all $y\in U_{r'}\setminus \overline{U_r}$. Therefore $f$ is continuous at $x$.

Next suppose that $f(x)=1$. For every $\varepsilon>0$, since $Q$ is dense in $[0, 1]$ we can find an $r$ in $Q$ such that $1-\varepsilon<r<1$. Then $x\in X\setminus \overline{U_r}$. Note that $r\le f(y)$ for all $y\in X\setminus \overline{U_r}$. Therefore $f$ is continuous at $x$.

Finally, suppose that $f(x)=0$. For every $\varepsilon>0$, since $Q$ is dense in $[0, 1]$ we can find an $r$ in $Q$ such that $0<r<\varepsilon$. Then $x\in U_r$. Note that $f(y)\le r$ for all $y\in U_r$. Therefore $f$ is continuous at $x$.
In summary, we have shown that $f$ is continuous on all of $X$.

Let $t\in (0, 1)$. When $n$ is sufficiently large, by (5) we can find some consecutive $r_1<r_2<r_3$ in $Q_n$ such that $r_1<t<r_3$.  Then $f^{-1}(t)$ is disjoint from $\overline{U_{r_1}}$, and $f^{-1}(t)\subseteq U_{r_3}$. Thus for every $\mu\in \sN$ we have
\begin{align*}
\mu(\{ f^{-1}(t) \})
\le \mu(U_{r_3}\setminus \overline{U_{r_1}})
&=\mu(U_{r_3}\setminus \overline{U_{r_2}})+\mu(\partial U_{r_2})+\mu(U_{r_2}\setminus \overline{U_{r_1}}) \\
&<\frac{1}{2^n}+0+\frac{1}{2^n}=\frac{1}{2^{n-1}}.
\end{align*}
Letting $n\to \infty$, we see that $f^{-1}(t)$ is $\sN$-small.

For each $n$, denoting by $r$ the largest element in $Q_n\setminus \{1\}$ we have $f^{-1}(0)\setminus A_0\subseteq U_{\min Q_n}\setminus U_0$ and $f^{-1}(1)\setminus A_1\subseteq U_1\setminus \overline{U_r}$, so that by (5) we have, for each $\mu\in \sN$, 
\begin{gather*}
\mu(f^{-1}(0)\setminus A_0)\le \mu(U_{\min Q_n}\setminus U_0)<\frac{1}{2^n}
\end{gather*}
and
\begin{gather*}
\mu(f^{-1}(1)\setminus A_1)\le \mu(U_1\setminus \overline{U_r})<\frac{1}{2^n}.
\end{gather*}
Letting $n\to \infty$, we conclude that $f^{-1}(0)\setminus A_0$ and $f^{-1}(1)\setminus A_1$ are $\sN$-small.

Now we turn to the construction of the sequence $\{Q_n\}_{n\in \Nb}$ and the sets $U_r$ satisfying (1) to (5). Set $U_0=A_0$.

By Lemma~\ref{L-SBP to function2} we can find an integer $M\ge 2$ and open sets $U_r$ for 
$r\in Q_1 :=\{\frac{1}{M}, \frac{2}{M}, \dots, \frac{M}{M}\}$ such that the following hold:
\begin{enumerate}
\item[(i)] $\overline{U_{j/M}}\subseteq U_{(j+1)/M}$ for all $j=0, 1, \dots, M-1$,
\item[(ii)] $U_1=X\setminus A_1$,
\item[(iii)] $\partial U_{j/M}$ is $\sN$-small for all $j=1, \dots, M-1$,
\item[(iv)] $\mu(U_{(j+1)/M}\setminus \overline{U_{j/M}})<2^{-1}$ for all $\mu\in \sN$ and $j=0, 1, \dots, M-1$.
\end{enumerate}
This defines $Q_1$ and $U_r$ for $r\in \{0\}\cup Q_1$ satisfying (1) to (5).

Let $r_0<r_1$ be two consecutive elements in $\{0\}\cup Q_1$. By Lemma~\ref{L-SBP to function2} we can find an integer $M\ge 2$ and open sets $U_r$ for $r\in Q_{2, r_0} :=\{r_0+(r_1-r_0)/M, r_0+2(r_1-r_0)/M, \dots, r_0+M(r_1-r_0)/M\}$ such that the following hold:
\begin{enumerate}
\item[(i)] $\overline{U_{r_0+j(r_1-r_0)/M}}\subseteq U_{r_0+(j+1)(r_1-r_0)/M}$ for all $j=0, 1, \dots, M-1$,
\item[(ii)] $\partial U_{r_0+j(r_1-r_0)/M}$ is $\sN$-small for all $j=1, \dots, M-1$,
\item[(iii)] $\mu(U_{r_0+(j+1)(r_1-r_0)/M}\setminus \overline{U_{r_0+j(r_1-r_0)/M}})<2^{-2}$ for all $\mu\in \sN$ and $j=0, 1, \dots, M-1$.
\end{enumerate}
Note that $U_{r_0+M(r_1-r_0)/M}=U_{r_1}$ was chosen before already.

Now set $Q_2=\bigcup_{r\in (\{0\}\cup Q_1)\setminus \{1\}}Q_{2, r}$. This defines $Q_2$ and $U_r$ for $r\in \{0\}\cup Q_2$ satisfying (1) to (5).

Let $r_0<r_1$ be two consecutive elements in $\{0\}\cup Q_2$. By Lemma~\ref{L-SBP to function2} we can find an integer $M\ge 2$ and open sets $U_r$ for $r\in Q_{3, r_0}:=\{r_0+(r_1-r_0)/M, r_0+2(r_1-r_0)/M, \dots, r_0+M(r_1-r_0)/M\}$ such that the following hold:
\begin{enumerate}
\item[(i)] $\overline{U_{r_0+j(r_1-r_0)/M}}\subseteq U_{r_0+(j+1)(r_1-r_0)/M}$ for all $j=0, 1, \dots, M-1$,
\item[(ii)] $\partial U_{r_0+j(r_1-r_0)/M}$ is $\sN$-small for all $j=1, \dots, M-1$,
\item[(iii)] $\mu(U_{r_0+(j+1)(r_1-r_0)/M}\setminus \overline{U_{r_0+j(r_1-r_0)/M}})<2^{-3}$ for all $\mu\in \sN$ and $j=0, 1, \dots, M-1$.
\end{enumerate}
Note that $U_{r_0+M(r_1-r_0)/M}=U_{r_1}$ was chosen before already.

Now set $Q_3=\bigcup_{r\in (\{0\}\cup Q_2)\setminus \{1\}}Q_{3, r}$. This defines $Q_3$ and $U_r$ for $r\in \{0\}\cup Q_3$ satisfying (1) to (5).

Continuing in this way, we find $\{Q_n\}_{n\in \Nb}$ and $U_r$ satisfying (1) to (5).
\end{proof}

\section{Product actions and the proof of Theorem~\ref{T-main}} \label{S-product actions}

Let $G$ and $H$ be countable groups. The product $G\times H\curvearrowright X\times Y$
of two actions $G\curvearrowright X$ and $H\curvearrowright Y$ is defined by $(g,h)(x,y) = (gx,hy)$. 

Let $G\curvearrowright X$ be a continuous action on a compact metrizable space.
Write $\sM_G (X)$ for the space of $G$-invariant Borel probability measures on $X$, which is a weak$^*$ closed
subset of the space $\sM (X)$ of all Borel probability measures on $X$.
The action $G\curvearrowright X$ is said to have 
the {\it SBP} if $X$ has the $\sM_G(X)$-SBP in the sense of Definition~\ref{D-small}.

The next lemma makes use of the well-known comparison property for ergodic
probability-measure-preserving actions $G\curvearrowright (X,\mu )$, which says that for any two 
sets $A,B\subseteq X$ of equal measure there exists an element $T$ of the full group 
(i.e., a measurable automorphism $T$ of $X$ for which there are a countable measurable partition 
$\{ A_i \}_{i\in I}$ of $X$ modulo null sets and group elements $s_i$ for $i\in I$ with $Tx = s_i x$
for all $i\in I$ and $x\in A_i$) such that $TA = B$ modulo null sets (see for example Lemma~7.10 of \cite{KecMil04}).
This is easily deduced using a greedy algorithm with respect to some fixed enumeration of $G$.

\begin{lemma}\label{L-ergodic for product}
Let $G\curvearrowright X$ and $H\curvearrowright Y$ be continuous actions on compact metrizable spaces. Then every ergodic measure $\mu \in \sM_{G\times H} (X\times Y)$ for the product action $G\times H\curvearrowright X\times Y$ is of the form $\mu_X\times \mu_Y$ for some $\mu_X\in \sM_G(X)$ and $\mu_Y\in \sM_H(Y)$.
\end{lemma}

\begin{proof}
Let $\mu$ be an ergodic measure in $\sM_{G\times H} (X\times Y)$ and write $\mu_X$ and $\mu_Y$
for the marginals, which are invariant and ergodic for the actions of $G$ and $H$, respectively.
If neither $\mu_X$ and $\mu_Y$ are atomless then by ergodicity they are each supported 
and uniform on a single finite orbit, which means that $\mu$ must be supported 
and uniform on the product of these finite orbits and therefore be equal to $\mu_X \times \mu_Y$.
We may thus assume that one of $\mu_X$ and $\mu_Y$ is atomless, say $\mu_X$ without loss of generality.

Let $A\subseteq X$ and $B\subseteq Y$ be Borel sets. 
Suppose first that $\mu_X (A)$ is rational, say $k/n$. 
The fact that $\mu$ is atomless and $X$ is metrizable implies that $(X,\mu )$ is measure-isomorphic 
to the unit interval with Lebesgue measure
\cite[Theorem~A.20]{KerLi16} and so we can find a $\mu_X$-uniform Borel partition $\{ P_1 , \dots ,  P_n \}$ of $X$
such that $A = \bigsqcup_{i=1}^k P_i$.
Since $G\curvearrowright X$ is $\mu_X$-ergodic, by the comparison property 
recalled before Lemma~\ref{L-ergodic for product}
there exist elements $S_2 , \dots , S_n$ in the full group
such that, modulo $\mu_X$-null sets, we have $P_i = S_i P_1$ for every $i=2,\dots , n$. 
For $i=2,\dots ,n$ the product transformation $T_i := S_i \times \id_Y$
of $X\times Y$, defined modulo a $\mu$-null set, belongs to the full group of the $G\times H$ action, 
and so all of the sets
$P_i \times B$ have the same $\mu$-measure. Since these sets partition $X\times B$, which has $\mu$-measure
$\mu_Y (B)$, we therefore have $\mu (P_i \times B) = \mu_Y (B) / n$ for every $i$. 
It follows that
\begin{align*}
\mu (A\times B) 
= \sum_{i=1}^k \mu (P_i \times B) 
= \sum_{i=1}^k \frac1n \mu_Y (B) 
= \frac{k}{n} \mu_Y (B) 
= \mu_X (A) \mu_Y (B) .
\end{align*}

In the case that $\mu_X (A)$ is not rational we can find, by the atomlessness of $\mu_X$, an increasing
sequence $A_1 \subseteq A_2 \subseteq\dots$ of subsets of $A$ with rational $\mu_X$-measure such that 
$A = \bigcup_{j=1}^\infty A_j$,
in which case $\mu_X (A) = \lim_{j\to\infty} \mu_X (A_j )$ and hence, using the previous paragraph,
\begin{align*}
\mu (A\times B) 
= \lim_{j\to\infty} \mu (A_j \times B) 
= \lim_{j\to\infty} \mu_X (A_j ) \mu_Y (B) 
= \mu_X (A) \mu_Y (B) .
\end{align*}
Since Borel rectangles generate the Borel $\sigma$-algebra of $X\times Y$, we conclude that $\mu$ 
is equal to the product measure $\mu_X \times \mu_Y$.
\end{proof}

\begin{lemma} \label{L-SBP to nonatomic}
Let $G\curvearrowright X$ and $H\curvearrowright Y$ be continuous actions on compact metrizable spaces. Suppose that the product action $G\times H\curvearrowright X\times Y$ has the SBP. Then either $H\curvearrowright Y$ has the SBP or $\sM_G(X)$ consists of atomless measures.
\end{lemma}

\begin{proof} 
Suppose that $H\curvearrowright Y$ does not have the SBP and there is a $\mu\in \sM_G(X)$ which is not atomless. 
By Proposition~\ref{P-small separation} we can find distinct $y_1, y_2\in Y$ such that for any open neighbourhood $U$ of $y_1$ satisfying $y_2\not\in \overline{U}$, the boundary $\partial U$ is not $\sM_H(Y)$-small. Also, we can find an $x\in X$ such that $\mu(\{ x \})>0$.

Since $G\times H\curvearrowright X\times Y$ has the SBP, we can find an open neighbourhood $O$ of $(x, y_1)$ in $X\times Y$ such that $(x, y_2)\not\in \overline{O}$ and $\partial O$ is $\sM_{G\times H}(X\times Y)$-small. Denote by $U$ the image of $O\cap (\{x\}\times Y)$ under the projection $X\times Y\rightarrow Y$. Then $U$ is an open neighbourhood of $y_1$, and $y_2\not\in \overline{U}$, and  $\{x\}\times \partial U\subseteq \partial O$. 
By our choice of $y_1$ and $y_2$, there exists
a $\nu\in \sM_H(Y)$ such that $\nu(\partial U)>0$. Now $\mu\times \nu\in \sM_{G\times H}(X\times Y)$, and
\begin{gather*}
(\mu\times \nu)(\partial O)\ge (\mu\times \nu) (\{x\}\times \partial U)=\mu(\{ x\})\nu(\partial U)>0,
\end{gather*}
a contradiction. Therefore either $H\curvearrowright Y$ has the SBP or $\sM_G(X)$ consists of atomless measures.
\end{proof}

Note that $\sM_G(X)$ consists of atomless measures if and only if $G\curvearrowright X$ has no finite orbits. 

\begin{proof}[Proof of Theorem~\ref{T-main}] 
(1)$\Rightarrow$(2). Let $H\curvearrowright Y$ be a continuous action of a countable group on a compact metrizable space, and let $(x_1, y_1)$ and $(x_2, y_2)$ be two distinct points in $X\times Y$. To show (2), by Proposition~\ref{P-small separation} it suffices to find an open set $O$ of $X\times Y$ such that $(x_1, y_1)\in O$, $(x_2, y_2)\not\in \overline{O}$, and $\partial O$ is $\sM_{G\times H}(X\times Y)$-small.

Consider first the case $x_1\neq x_2$. Since $G\curvearrowright X$ has the SBP, we can find an open set $U\subseteq X$ such that $x_1\in U$, $x_2\not\in \overline{U}$, and $\partial U$ is $\sM_G(X)$-small. Set $O=U\times Y$. Then $(x_1, y_1)\in O$ and $(x_2, y_2)\not\in \overline{O}$. Since $\partial O=\partial U\times Y$ and for every $\mu\in M_{G\times H}(X\times Y)$
the $\mu$-measure of this set is equal to the value of $\partial U$ on the marginal measure, which 
is invariant under $G$, we know that $\partial O$ is $\sM_{G\times H}(X\times Y)$-small.

Next consider the case $x_1=x_2$. Fix a compatible metric $\rho$ on $Y$ such that $\rho(y_1, y_2)>2$. For each $r\ge 0$ denote by $S_r$ the set $\{y\in Y: \rho(y, y_1)=r\}$. By assumption the action $G\curvearrowright X$ does not have finite orbits,
and so $\sM_G (X)$ consists of atomless measures, which means that we can apply Proposition~\ref{P-SBP to function},
with both $A_0$ and $A_1$ taken there to be the empty set, in order
to obtain a continuous function $f: X\rightarrow [0, 1]$ such that $f^{-1}(t)$ is $\sM_G(X)$-small for all $t\in [0, 1]$. Put $g=f+1$. Set
\begin{gather*}
O=\{(x, y)\in X\times Y: \rho(y, y_1)<g(x)\}.
\end{gather*}
Then $O$ is open in $X\times Y$, $(x_1, y_1)\in O$, and $(x_2, y_2)\not\in \overline{O}$. Note that
\begin{gather*}
\partial O \subseteq \{(x, y)\in X\times Y: \rho(y, y_1)=g(x)\}.
\end{gather*}
To show that $\partial O$ is $\sM_{G\times H}(X\times Y)$-small, it suffices 
to show that $\mu(\partial O)=0$ for every ergodic $\mu$ in $\sM_{G\times H}(X\times Y)$. 
Indeed if we knew the latter and $\mu$ were an arbitrary measure in $\sM_{G\times H}(X\times Y)$ then 
the ergodic disintegration $\int_Z \mu_z \, d\nu (z)$ of $\mu$ \cite[Theorem~8.7]{Gla03} 
would yield $\mu (\partial O) = \int_Z \mu_z (\partial O) \, d\nu (z) = 0$.

Let $\mu\in \sM_{G\times H}(X\times Y)$ be ergodic. By Lemma~\ref{L-ergodic for product} we have $\mu=\mu_X\times \mu_Y$ for some $\mu_X\in \sM_G(X)$ and $\mu_Y\in \sM_H(Y)$. Denote by $W$ the set of all $r\ge 0$ such that $\mu_Y(S_r)>0$. This set is countable since
the sets $S_r$ for $r\ge 0$ are pairwise disjoint, and so $\mu_X(g^{-1}(W))=0$. For each $x\in X$, denote by $(\partial O)_x$ the image of $\partial O\cap (\{x\}\times Y)$ under the projection $X\times Y\rightarrow Y$. Then
$$ \partial O=\bigcup_{x\in X}(\{x\}\times (\partial O)_x)\subseteq \bigcup_{x\in X} (\{x\}\times S_{g(x)}).$$
By Fubini's theorem,
\begin{align*}
\mu(\partial O)=\int_X \mu_Y((\partial O)_x)\, d\mu_X(x)=\int_{g^{-1}(W)}\mu_Y((\partial O)_x)\, d\mu_X(x)=0.
\end{align*}

(2)$\Rightarrow$(1). Take a continuous action $H\curvearrowright Y$ of a countable group on a compact metrizable space
such that the action does not have the SBP and 
$\sM_H(Y)$ does not consist entirely of atomless measures. For example, one may take the shift action 
$\Zb\curvearrowright [0, 1]^\Zb$, which has nonzero mean dimension \cite[Proposition~3.3]{LinWei00}
and hence fails to have the SBP \cite[Theorem~5.4]{LinWei00}.
By (2) the product action $G\times H\curvearrowright X\times Y$ has the SBP. It follows from Lemma~\ref{L-SBP to nonatomic} that (1) holds.
\end{proof}

\section{SBP pairs} \label{S-SBP pairs}

Let $G\curvearrowright X$ be a continuous action of a countable group a compact metrizable space. 
Set $\Delta(X)=\{(x, x): x\in X\}\subseteq X^2$.

Motivated by Proposition~\ref{P-small separation}, we make the following definition.

\begin{definition} \label{D-SBP pair}
We say that a pair $(x_1, x_2)\in X^2\setminus \Delta(X)$ is an {\it SBP pair} if there is no open neighbourhood $U$ of $x_1$ such that $x_2\not\in \overline{U}$ and $\partial U$ is $\sM_G(X)$-small. Denote by ${\rm SBP}_2(X)$ the set of all SBP pairs of $X$.
\end{definition}

For disjoint closed subsets $A_1$ and $A_2$ of $X$, we say that $A_1$ and $A_2$ are {\it separated by closed $\sM_G(X)$-small sets} if there is a partition $\{ U_1, Y, U_2 \}$ of $X$ such that $U_j$ is an open set with $A_j\subseteq U_j$ for $j=1, 2$ and $Y$ is $\sM_G(X)$-small. The following is straightforward.

\begin{lemma} \label{L-SBP split}
Let $A_1$ and $A_2$ be disjoint closed subsets of $X$ that are not separated by closed $\sM_G(X)$-small sets. Let $A_{1, 1}$ and $A_{1, 2}$ be closed subsets of $A_1$ such that $A_1=A_{1, 1}\cup A_{1, 2}$. Then either $A_{1, 2}$ and $A_2$ are not separated by closed $\sM_G(X)$-small sets or $A_{1, 2}$ and $A_2$ are not separated by closed $\sM_G(X)$-small sets.
\end{lemma}

\begin{proposition} \label{P-SBP basic}
The following hold.
\begin{enumerate}
\item $G\curvearrowright X$ does not have the SBP if and only if ${\rm SBP}_2(X)\neq \emptyset$.
\item ${\rm SBP}_2(X)\cup \Delta(X)$ is closed and $G$-invariant.
\item If $A_1$ and $A_2$ are disjoint closed subsets of $X$ such that $A_1$ and $A_2$ are not separated by closed $\sM_G(X)$-small sets, then $(A_1\times A_2)\cap {\rm SBP}_2(X)\neq \emptyset$.
\item Let $\pi:X\rightarrow Y$ be a factor map. Then $\pi\times \pi({\rm SBP}_2(X))\subseteq {\rm SBP}_2(Y)\cup \Delta(Y)$.
\end{enumerate}
\end{proposition}
\begin{proof} 
The first statement is part of Proposition~\ref{P-small separation}, while (2) and (4) are obvious and (3) follows from Lemma~\ref{L-SBP split}.
\end{proof}

When $G$ is sofic and $\Sigma$ is a sofic approximation sequence for $G$, Garc\'{\i}a-Ramos and Gutman defined a pair $(x_1, x_2)\in X^2\setminus \Delta(X)$ to be a sofic mean dimension pair if for any disjoint closed neighbourhoods $A_1$ and $A_2$ of $x_1$ and $x_2$, respectively, one has ${\rm mdim}_\Sigma(\cU)>0$ for $\cU=\{X\setminus A_1, X\setminus A_2\}$ \cite[Definition 4.1]{GarGut24}. When $G$ is amenable, one can also define mean dimension pairs using ${\rm mdim}(\cU)$,
i.e., by averaging over F{\o}lner sets instead of a sofic approximation sequence. When $G$ is amenable and infinite, using \cite[Lemma 3.5]{Li13} and the proof of \cite[Lemma 3.7]{Li13} one can see that these two definitions are equivalent.

In the case that $G$ is amenable and infinite, Lindenstrauss and Weiss showed that if $G\curvearrowright X$ has the SBP then it has mean dimension zero \cite[Theorem 5.4]{LinWei00}. Their argument also works here.

\begin{theorem} \label{T-mdim pair to SBP pair}
Suppose that $G$ is amenable and infinite. Then every mean dimension pair is an SBP pair.
\end{theorem}

\begin{question} \label{Q-SBP pair to mdim pair}
For free minimal actions of $\Zb^d$, must SBP pairs be mean dimension pairs?
\end{question}

\end{document}